\documentclass[11pt,a4paper]{article}
\usepackage[utf8]{inputenc}
\usepackage{amsmath,amssymb,amsthm,mathtools}
\usepackage{graphics} 
\usepackage{graphicx}  
\usepackage{epstopdf}
\usepackage[colorlinks=true]{hyperref}
\usepackage{tikz}
\usepackage{makeidx}
\usepackage{subfigure}
\usepackage{caption}
\usepackage{setspace}
\usepackage{a4wide}

\theoremstyle{plain}
\newtheorem{theorem}{Theorem}[section]

\newtheorem{lemma}[theorem]{Lemma}

\theoremstyle{definition}

\theoremstyle{remark}
\newtheorem{remark}[theorem]{Remark}

\newcommand{\R}{\mathbb{R}}
\newcommand{\Dt}{\Delta t}
\newcommand{\Dx}{\Delta x}

\begin{document}
 
\title{Properties of the LWR model with time delay}

\author{Simone G\"ottlich\footnotemark[1], \; Elisa Iacomini\footnotemark[1], \; Thomas Jung\footnotemark[4]}

\footnotetext[1]{University of Mannheim, Department of Mathematics, 68131 Mannheim, Germany (goettlich@uni-mannheim.de, eiacomin@mail.uni-mannheim.de)}
\footnotetext[4]{Fraunhofer Institute ITWM, 67663 Kaiserslautern, Germany (thomas.jung@itwm.fraunhofer.de)}

\date{ \today }
\maketitle

\begin{abstract}
\noindent
In this article, we investigate theoretical and numerical properties of the first-order Lighthill-Whitham-Richards (LWR) traffic flow model with time delay. Since standard results from the literature are not directly applicable to the delayed model, we mainly focus on the numerical analysis of the proposed finite difference discretization. The simulation results also show that the delay model is able to capture Stop \& Go waves. 
\end{abstract}




\section{Introduction}

Nowadays traffic models have become an indispensable tool in the urban and extraurban management of vehicular traffic.
Understanding and developing an optimal transport network, with efficient
movement of traffic and minimal traffic congestions, will have a great socio- economical impact on the society. This is why  in the last decades an intensive research activity in the field of traffic flow modelling flourished.

Literature about traffic flow is quite large and many methods have been developed resorting to different approaches.
Starting from the natural idea of tracking every single vehicle, several microscopic models, based on the idea of Follow-the-Leader, grew-up for computing positions, velocities and accelerations of each car by means of systems of ordinary differential equations (ODEs) \cite{aw2002SIAP,braskstone2000car,cristiani2016NHM,difrancesco2017MBE,treiber2013traffic}. Other ways go from kinetic \cite{herty2020bgk,puppo2016kinetic,tosin2017MMS} to macroscopic fluid-dynamic and measures approaches \cite{aw2000SIAP,cacace2018measure,camilli2018measure,piccolibook,Piccoli2012review,rosini2013BOOK}, focusing on averaged quantities, such as 
the traffic density and the speed of the traffic flow, by means of systems of hyperbolic partial differential equations (PDEs), in particular conservation laws. In this way we loose the detailed level of vehicles' description, indeed they become indistinguishable from each other. 
The choice of the scale of observation mainly depends on the number of the involved vehicles, the size of the network and so on.

In this paper we deal with the macroscopic scale, in particular we will focus on first order macroscopic models. The most relevant model in this framework is the LWR model, introduced by Lighthill, Whitham \cite{lighthill1955RSL} and Richards \cite{richards1956OR} in the '50. The main idea underlying this approach is that the total mass has to be preserved, since cars can not disappear. Moreover, in this model the mean velocity is supposed to be dependent on the density, thus is closing the equation. 
On the other hand the lacks of the LWR model are well-known. For example, it fails to generate capacity drop, hysteresis, relaxation, platoon diffusion, or spontaneous congestions like Stop \& Go waves, that are typical features of traffic dynamics.
These drawbacks are due to the fact that the LWR model represents a simplification of the reality, assuming that accelerations are instantaneous and traffic is described only at the equilibrium.

In order to overcome these issues, second order models have been proposed, see \cite{aw2002SIAP,aw2000SIAP, zhang2002non}. They take into account
the non-equilibria states, assuming that accelerations are not instantaneous. To do this, the equation that describes the variation of the velocity in time has to be added to the system, replacing the typical given law of the fist order models.
Other ways are also possible to improve first order models, just considering phase transition models \cite{blandin2011general,colombo2003hyperbolic} or multi-scale approaches \cite{colombo2015mixed,cristiani2019interface}.
Instead of switching to second order models, we propose a first order macroscopic model with a time delay term in the flux function, for taking into account that the velocity can not change instantaneously. In this framework the delay represents the reaction time of both drivers and vehicles.

At a microscopic level, a model with time delay appears for the first time in the work done by Newell \cite{newell1961nonlinear}, then similar models are presented in \cite{bando1995dynamical,chandler1958traffic}. The mathematical tools needed in this framework are not systems of ODEs anymore, but systems of delay differential equations (DDEs), particular differential equations in which the derivative of the unknown function at a certain time is given in terms of the values of the function at previous times.
Macroscopic models can be derived from microscopic description following a well-known procedure described in \cite{aw2002SIAP,colombo2014micro,difrancesco2017MBE}.
Depending on how to treat the delay term, one can recover different macroscopic models, as in \cite{tordeux2018traffic}, in which a Taylor's approximation is applied to the delay term and the obtained model is a diffusive LWR type model. 
On the other hand, we want to keep the delay in the explicit form, and therefore avoid the diffusion approximation. The model derived in \cite{burger2018derivation} will be studied in details in the following, investigating carefully its theoretical and numerical features. 

Several delayed-systems are presented in literature, since many phenomena need some transient to become visible or effective: the study of the evolution of the HIV in medicine \cite{culshaw2000delay,smith2011introduction}, cell population dynamics in biology \cite{green1981diffusion,rey1993multistability}, the feedback control loops in control engineering \cite{kwon1980feedback}, and many applications in mechanics and economics \cite{bianca2013time}, but to the authors best knowledge, they are closer to delayed parabolic partial differential equations or to delayed ordinary differential equations, i.e.\ they are studied only at a microscopic level.

In this work instead we deal with a delayed hyperbolic partial differential equation. We will point out similarities and differences with the undelayed model in order to catch the effect of the delay on traffic dynamics, both from theoretical and numerical points of view.
Moreover, since we are interested in reproducing real traffic phenomena, the numerical tests are mainly focused on traffic instabilities.
In particular we investigate the phenomenon of Stop \& Go (S\&G) waves, which are a typical feature of congested traffic and represent a real danger for drivers. 
They lead not only to safety hazard, but they also have a negative impact on fuel consumption and pollution.
Indeed a S\&G wave is detected when vehicles stop and restart without any apparent reason, generating a wave that travels backward with respect to the cars' trajectories.
Since modeling properly this phenomenon is crucial for developing techniques
aimed at reducing it, a considerable literature is growing up on this topic. This means that a lot of models have been developed in the last years, i.e. \cite{borsche2018nonlinear,flynn2009PRE,herty2020bgk,hoogendoorn2014ITS,tordeux2018traffic}, 
and also several real experiments took place, just see \cite{piccoli2018TRC,zhao2017TRB}.

In this framework, our aim is to investigate if our delayed model is able to capture the S\&G phenomena and, therefore, to present an easy to use algorithm able to reproduce S\&G waves at a macroscopic level. 
Indeed from the numerical point of view, just an altered Lax Friedrichs method will be employed to compute the evolution of the density.
In order to validate our model, several numerical tests will be provided for
comparing our delayed model with the existing ones.

\subsubsection*{Paper organization}
In Section 2, we introduce the delayed model and investigate its theoretical properties, as the conservation of mass, the positivity and the boundedness of the solution. After that, we focus on the numerical aspects, presented in Section 3, proposing a suitable numerical scheme and checking the theoretical features still hold. Section 4 is completely devoted to the numerical tests.

\section{The delayed traffic flow model}

In macroscopic models \cite{garavello2006traffic}, traffic is described in terms of macroscopic variables such as density $\rho=\rho(x,t)$, that is the number of vehicles in a kilometre, and the mean velocity $V=V(x,t)$ at the point $x\in \R$ at time $t>0$. 

The LWR model, introduced by Lighthill, Whitham \cite{lighthill1955RSL} and Richards \cite{richards1956OR}, is one of
the oldest and still most relevant first order macroscopic models for traffic flow.
The natural assumption that the total mass is conserved along the road is closed by the assumption that the velocity $V=V(\rho)$ is given as function of the density $\rho$:
\begin{equation}\label{LWR}
\begin{cases}
\partial_t\rho(x,t)+\partial_x(\rho(x,t) \ V(\rho(x,t)))=0\\
\rho(x,0)=\rho^0(x).
\end{cases}
\end{equation}
A lot of possible choices for the function $V(\rho)$ are available in the literature, i.e.\ the Greenshields function \cite{greenshields1935study} which proposes a linear velocity function:
\begin{equation}
    V(\rho)=\rho_{max}\Big( 1- \frac{\rho}{\rho_{max}}\Big).
\end{equation}
In order to simplify the notation, we will consider the normalized quantities  $\rho_{max}=V_{max}=1$.
%
%
%
%
Aiming to overcome the drawbacks of LWR model presented in the introduction, we propose a first order macroscopic model with a time delay term in the flux function, for taking into account that the velocity can not change instantaneously.

In this framework the delay represents the reaction time of both drivers and vehicles.
Such a model has been recovered from a delayed microscopic model, as shown in \cite{burger2018derivation} keeping the delay in the explicit form. Assuming $T\ge0$ as the time delay, we consider:
\begin{equation}\label{eq:DLWR}
    \partial_t \rho(x,t) + \partial_x \left(\rho(x,t) \ V(\rho(x,t-T))\right)=0.
\end{equation}
We will call this model delayed LWR model. Note that in the limit case of $T=0$ the classical LWR model is recovered and therefore, it can be interpreted as a generalization of the LWR model. On the other hand,
if the delay is too large and there are suitable initial conditions, cars can overtake or crash each other, for example when a vehicle suddenly
brakes and the following car is not reacting in time to slow down.

Note that in order to guarantee the well-posedness of the problem, we have to provide an initial history function as initial data defined on $[-T,0]$, thus we need $\rho^0(x,t)$ defined on $t\in [-T,0]$ when starting at $t=0$.

\subsection{Properties of the model}
After introducing the delayed model, we want to investigate its properties. Since this model can be seen as a generalization of the classical LWR model, i.e. when $T=0$, it is natural to investigate how its properties differ from
the undelayed model.

\subsubsection{Conservation of mass}
In  the  framework  of  conservation  laws  and  traffic flow models the conservation of the total mass is a crucial property which has to be guaranteed. 
For the LWR model \eqref{LWR}, we have one equation and one
conserved quantity, i.e.\ $\rho$. Indeed cars do not appear or disappear, they can only enter and leave at the boundaries.
We note that the structure of the equation in the delayed and undelayed model stays
the same and that we have a flux function that is now dependent on two variables.

\begin{lemma}
The delayed LWR model \eqref{eq:DLWR} conserves the quantity $\rho(x,t)$.
\end{lemma}
\begin{proof}
	We integrate the equation \eqref{eq:DLWR} over an arbitrary space interval $[a,b]$ and get
	\begin{align*}
	\frac{d}{dt}\int_{a}^{b}\rho(x,t)dx&=-\int_{a}^{b}\partial_x(\rho(x,t)V(\rho(x,t-T)))dx\\
	&=\rho(a,t)V(\rho(a,t-T))-\rho(b,t)V(\rho(b,t-T)).
	\end{align*}
	Since $\int_{a}^{b}\rho(x,t)dx$ is the amount of density in the interval $[a,b]$, $\frac{d}{dt}\int_{a}^{b}\rho(x,t)dx$ denotes the change over time for the density. Therefore, the density only changes due to the flux at the boundaries $a$ and $b$ for every space interval.
\end{proof}

We see that the density is still conserved in the delayed model, which is very important for its reliability. The introduction of an explicit time delay therefore does not destroy this property.

\subsubsection{Positivity}
Another property one would ensure is the positivity of the solution. Indeed we want the density to stay positive, as negative densities have no physical meaning. 
\begin{lemma} \label{lem_pos}
	Assume we have initial data with non-negative density $\rho$. For the delayed LWR model, then the density stays non-negative.
\end{lemma}
\begin{proof}
	We rewrite \eqref{eq:DLWR} as
	$$\partial_{t}\rho(x,t)=-\left(\rho(x,t)\partial_xV(\rho(x,t-T))+V(\rho(x,t-T))\partial_x\rho(x,t)\right).$$
	For the density to become negative, we need to have $\rho=0$ and $\partial_t\rho<0$. Hence, assume $\rho(x^*,t^*)=0$ at an arbitrary point $(x^*,t^*)$. This means, due to our assumption, that $\rho(x^*,t^*)$ is a minimum, since the density is positive for $t\le t^*$ and therefore, if the derivative exists, $\partial_x\rho(x^*,t^*)=0$. Plugging in then gives us
	$$\partial_t\rho(x^*,t^*)=-\big(0\ \partial_xV(\rho(x^*,t^*-T))+V(\rho(x^*,t^*-T))\ 0\big)=0.$$
	We have therefore shown that $\rho$ can not become negative.
\end{proof}
\begin{remark}
	The velocity $V$ in this model is a function of $\rho$ and can be chosen and altered depending on the needs. The properties regarding the velocity in the first order model can therefore be acquired by choosing a suitable function $V$.  For example, we can have lower and upper bounds for the velocity by defining $V$ to be cut at the bounds.
\end{remark}

\subsubsection{Upper bound}

The last property we want to investigate is the boundedness of the solution. In particular, we want to know if there is a maximal density. For the undelayed model, this is guaranteed. For the delayed model, we need to check if this still true. 
\begin{lemma}
	Assume $V$ is monotone decreasing and $V(\rho_{\text{max}})=0$ for $\rho_{\text{max}},$ the maximal density in the classical LWR model. The delayed first order model \eqref{eq:DLWR} has no maximal density $\rho_{{max}}$.
\end{lemma}
\begin{proof}
	Assume we have a maximal density $\rho_{{max}}$. The velocity function $V$ is chosen in such a way that $V(\rho_{{max}})=0$ and monotone decreasing. Then, for an arbitrary point $(x^*,t^*)$ where $\rho(x^*,t^*)=\rho_{{max}}$ we have
	$$\partial_t \rho(x^*,t^*)=-\rho(x^*,t^*)\partial_xV(\rho(x^*,t^*-T))-V(\rho(x^*,t^*-T))\partial_x\rho(x^*,t^*).$$
	Since $\rho(x^*,t^*)$ is the maximal density, $\partial_x\rho(x^*,t^*)=0$ if the derivative exists and we have
	$$\partial_t \rho(x^*,t^*)=-\rho(x^*,t^*)\partial_xV(\rho(x^*,t^*-T))$$
	left. We know $\rho(x^*,t^*)\geq0$ and this means that the sign of $\partial_t\rho(x^*,t^*)$ is only dependent on $\partial_xV(\rho(x^*,t^*-T))$.
	
	In the undelayed case, we know that $\partial_xV(\rho(x^*,t^*))>0$, since $V$ is monotone decreasing and $\rho(x^*,t^*)$ is the maximal $\rho$. 
	
	In the delayed case, we do not have knowledge if $\rho(x^*,t^*-T)$ is maximal, so we can in general say nothing about $\partial_xV(\rho(x^*,t^*-T))$. This means, in general, $\rho>\rho_{{max}}$ is possible.
\end{proof}
\begin{remark}
	Regarding the positivity, we claimed that the choice of $V$ in the first order model is a key to guarantee a positive velocity. We here see, due to the fact that $\rho$ overshoots any $\rho_{{max}}$, that the classical choices for $V$ must be altered to avoid negative velocities, i.e. we need to cut the function.
\end{remark}
\begin{remark}
If the density $\rho>\rho_{{max}}$, the model is not reliable any more. On the other hand this situation could not be avoided since rear-end collisions are actually possible in real situations.
\end{remark}

\section{Numerical discretization}

After the investigations on the analytical properties of the delayed LWR model, let us focus on its numerical counterpart.

Since \eqref{eq:DLWR} is a hyperbolic partial differential equation, we can employ the Lax-Friedrichs method for the numerical approximation.
To do that, we first introduce space and time steps $\Delta x$, $\Delta t >0$  and a grid in space $\{x_i= i \Dx , i \in \mathbb{Z}\}$ and time $\{t^n= n \Dt , n \in \mathbb{N}\}$.
Discretized variables are expressed by $\rho_i^n$, where $i$ is the space and $n$ the time index.  Also, we have $\Delta t \leq T$ to be able to treat the delay.

The Lax-Friedrichs method for \eqref{eq:DLWR} is stated by:
$$\rho_i^{n+1}=\frac{1}{2}(\rho_{i+1}^n+\rho_{i-1}^n)-\frac{\Delta t}{2\Delta x}(f(\rho_{i+1}^n)-f(\rho_{i-1}^n)).$$
Now using the structure of \eqref{eq:DLWR}, we can identify a flux function $f(\rho(x,t-T),\rho(x,t))=V(\rho(x,t-T))\rho(x,t)$. Plugging this into the Lax-Friedrichs method, we end up with an \textit{altered Lax-Friedrichs} method
\begin{equation}
\label{AlteredLF}
\rho_i^{n+1}=\frac{1}{2}(\rho_{i+1}^n+\rho_{i-1}^n)-\frac{\Delta t}{2\Delta x}(f(\rho_{i+1}^{n-T_{\Delta}},\rho_{i+1}^{n})-f(\rho_{i-1}^{n-T_{\Delta}},\rho_{i-1}^{n})),
\end{equation}
where $T_{\Delta}$ is the number of steps that make up the time delay $T$.
In order to guarantee the well-posedness of the discrete problem, we have to provide an initial history function as initial data defined on $[-T,0]$, as we said above for the continuous problem. The simplest choice one can do is to consider $\rho^0(x,t)$ as a constant function on $t\in [-T,0]$ when starting at $t=0$. In the following we will assume that $\rho^0(x,0)$ is constant in $t$ for $t\in [-T,0]$.
The Lax-Friedrichs method has a CFL condition in the classical case, which is given as $\Delta t\leq\frac{\Delta x}{\max_k(\lambda_k)}$, where $\lambda_k$ are the eigenvalues of the jacobian matrix of $f$. We also expect to find a CFL condition in the delayed case, but a priori it is not clear how this condition may look like. In the following, we want to investigate some properties of this method, and in this process we will find an appropriate CFL condition. For the sake of the calculations, we assume the velocity function $V$ to be the Greenshields function, or a cut variation of it, where we have $|V(\rho)|\leq|\rho_{\text{max}}|$ with $V_{max}=1$.

\subsection{Properties of the discretization}

\subsubsection{Conservation of mass}

First, we check if the conservation property is preserved from the numerical scheme. Here, we assume the density to be on a compact support, so we do not have infinite density initially.
We get
\begin{equation}
\label{Conservation}
\Delta x\sum_j \rho^{n+1}_j=\Delta x\sum_j \frac{1}{2}(\rho_{j+1}^{n}+\rho_{j-1}^{n})-\frac{\Delta t}{2\Delta x}(V(\rho_{j+1}^{n-T_{\Delta}})\rho_{j+1}^{n}-V(\rho_{j-1}^{n-T_{\Delta}})\rho_{j-1}^{n}),
\end{equation}
where the part $$\sum_j \frac{1}{2}(\rho_{j+1}^{n}+\rho_{j-1}^{n})=\sum_j \rho_j^n$$ and the part $$\sum_j -\frac{\Delta t}{2\Delta x}(V(\rho_{j+1}^{n-T_{\Delta}})\rho_{j+1}^{n}-V(\rho_{j-1}^{n-T_{\Delta}})\rho_{j-1}^{n})$$ is a telescope sum and equals zero due to the compact support. This gives us $$\Delta x\sum_j \rho^{n+1}_j=\Delta x\sum_j \rho^{n}_j$$ and therefore conservation.

\subsubsection{Positivity}

We show that Lemma \ref{lem_pos} holds also at the discrete level under a certain CFL condition. Starting with \eqref{AlteredLF}, we see
that $\frac{1}{2}(\rho_{j+1}^n+\rho_{j-1}^n)$ is always positive, since we assume $\rho^n$ to be positive. W.l.o.g. we can even say it is bigger than $0$, since if it is $0$, $\rho_{j}^{n+1}$ will also be zero. So to guarantee positivity, we need to guarantee
$$\frac{\Delta t}{2\Delta x}(V(\rho_{j+1}^{n-T_{\Delta}})\rho_{j+1}^{n}-V(\rho_{j-1}^{n-T_{\Delta}})\rho_{j-1}^{n})\leq\frac{1}{2}(\rho_{j+1}^{n}+\rho_{j-1}^{n}).$$
We introduce a CFL-condition, namely $\Delta t_n\leq\frac{\Delta x}{\max\{|\rho^n|,|\rho^{n-T_{\Delta}}|\}}$. Therefore, we get for the left-hand-side
\begin{align}
\label{Positivity}
&\frac{1}{2\max\{|\rho^n|,|\rho^{n-T_{\Delta}}|\}}(V(\rho_{j+1}^{n-T_{\Delta}})\rho_{j+1}^{n}-V(\rho_{j-1}^{n-T_{\Delta}})\rho_{j-1}^{n})\nonumber\\
\leq&\frac{V(\rho_{j+1}^{n-T_{\Delta}})}{2\max\{|\rho^n|,|\rho^{n-T_{\Delta}}|\}}\rho_{j+1}^{n}-\frac{V(\rho_{j-1}^{n-T_{\Delta}})}{2\max\{|\rho^n|,|\rho^{n-T_{\Delta}}|\}}\rho_{j-1}^{n}\nonumber\\
\leq&\frac{\max\{|\rho^n|,|\rho^{n-T_{\Delta}}|\}}{2\max\{|\rho^n|,|\rho^{n-T_{\Delta}}|\}}\rho_{j+1}^{n}+\frac{\max\{|\rho^n|,|\rho^{n-T_{\Delta}}|\}}{2\max\{|\rho^n|,|\rho^{n-T_{\Delta}}|\}}\rho_{j-1}^{n}\nonumber\\
=&\frac{1}{2}(\rho_{j+1}^{n}+\rho_{j-1}^{n}),
\end{align}
which shows the positivity for this CFL-condition. Here we use that $|V(\rho)|\leq\max\{|\rho^n|,|\rho^{n-T_{\Delta}}|\}$ if $\rho$ is positive.

\subsubsection{$L^\infty$-Bound}

Focusing on the boundedness of the discrete solution, we look for an estimate in the norm $|\cdot|_{L^{\infty}}$. Since we have positivity, only an upper bound for $\rho$ in \eqref{AlteredLF} is required. We assume that the data at time $t_n$ has an upper bound which we denote with $\rho^n_{\max}$. We further denote $V(\rho_j^{n-T_{\Delta}})\rho_j^n=f(\rho_j^{n-T_{\Delta}},\rho_j^n)$ and get for a special $\xi=(\xi_1,\xi_2)$ using the mean value theorem:
\begin{align}
\label{L-Bound}
\rho_j^{n+1}&=\frac{1}{2}(\rho_{j+1}^{n}+\rho_{j-1}^{n})-\frac{\Delta t}{2\Delta x}(f(\rho_{j+1}^{n-T_{\Delta}},\rho_{j+1}^{n})-f(\rho_{j-1}^{n-T_{\Delta}},\rho_{j-1})^{n})\nonumber\\
&= \frac{1}{2}(\rho_{j+1}^{n}+\rho_{j-1}^{n})+\frac{\Delta t}{2\Delta x}(\nabla f(\xi_1,\xi_2)\begin{pmatrix}\rho_{j-1}^n-\rho_{j+1}^n\\\rho_{j-1}^{n-T_{\Delta}}-\rho_{j+1}^{n-T_{\Delta}}\end{pmatrix})\nonumber\\
&=\frac{1}{2}(\rho_{j+1}^{n}+\rho_{j-1}^{n})+\frac{\Delta t}{2\Delta x}(V(\xi_2)(\rho_{j-1}^n-\rho_{j+1}^n)+\xi_1(\rho_{j+1}^{n-T_{\Delta}}-\rho_{j-1}^{n-T_{\Delta}}))\nonumber\\
&\leq\frac{1}{2}(\rho_{j+1}^{n}+\rho_{j-1}^{n})+\frac{1}{2}(\rho_{j-1}^n-\rho_{j-1}^{n-T_{\Delta}}+\rho_{j+1}^{n-T_{\Delta}}-\rho_{j+1}^n)\nonumber\\
&=\frac{1}{2}(2\rho_{j-1}^n-\rho_{j-1}^{n-T_{\Delta}}+\rho_{j+1}^{n-T_{\Delta}})\nonumber\\
&\leq\rho_{\max}^n+\frac{1}{2}(\rho_{j+1}^{n-T_{\Delta}}-\rho_{j-1}^{n-T_{\Delta}})\nonumber\\
&\leq\frac{3}{2}\max\{|\rho^n|,|\rho^{n-T_{\Delta}}|\}
\end{align}

We use again the positivity and the CFL-condition.\newline
\begin{remark}
We can find a different estimate for $\rho^{n+1}_j$, depending on $T$. Note that the estimate in \eqref{L-Bound} is more accurate. 
We start in the same way as above, and have 
$$\rho_j^{n+1}\leq\frac{1}{2}(2\rho_{j-1}^n-\rho_{j-1}^{n-T_{\Delta}}+\rho_{j+1}^{n-T_{\Delta}}).$$ Rearranging gives us
$$\frac{1}{2}(\rho_{j-1}^n+\rho_{j-1}^n-\rho_{j-1}^{n-T_{\Delta}}+\rho_{j+1}^{n-T_{\Delta}})=\frac{1}{2}(\rho_{j-1}^n-\rho_{j-1}^{n-T_{\Delta}}+\rho_{j+1}^{n-T_{\Delta}}+\rho_{j-1}^n),$$
and by again using the mean value theorem for $\rho_{j-1}(t)$ we get
\begin{equation}
\rho_j^{n+1}\leq\frac{1}{2}(T\partial_t\rho_{j-1}+\rho_{j-1}^n+\rho_{j+1}^{n-T_{\Delta}})\leq\max\{|\rho^n|,|\rho^{n-T_{\Delta}}|\}+\frac{1}{2}T||\partial_t\rho||_{L^{\infty}}
\end{equation}

\end{remark}

\subsubsection{TV Bound}

Let us look for an estimate on the Total Variation for the method \eqref{AlteredLF}. The first thing we check is the difference between the velocity function in two cells.
\begin{equation}
\label{DiffVel}
V(\rho_{j+1}^{n-T_{\Delta}})-V(\rho_{j}^{n-T_{\Delta}})\leq(1-\rho_{j+1}^{n-T_{\Delta}})-(1-\rho_{j}^{n-T_{\Delta}})=\rho_{j}^{n-T_{\Delta}}-\rho_{j+1}^{n-T_{\Delta}}=-\Delta_{j+\frac{1}{2}}^{n-T_{\Delta}}
\end{equation}
Then, we look at the difference between two neighboring cells, where we use \eqref{AlteredLF}, \eqref{DiffVel} and the notation $\Delta_{j+\frac{1}{2}}^{n}=\rho_{j+1}^{n}-\rho_{j}^{n}$:
\begin{align}
\label{Diff}
\Delta_{j+\frac{1}{2}}^{n+1}&=\rho_{j+1}^{n+1}-\rho_{j}^{n+1}\nonumber\\
&=\frac{1}{2}(\Delta_{j+\frac{3}{2}}^{n}+\Delta_{j-\frac{1}{2}}^{n})\nonumber\\
&-\frac{\Delta t}{2\Delta x}(V(\rho_{j+2}^{n-T_{\Delta}})\Delta_{j+\frac{3}{2}}^{n}-V(\rho_{j}^{n-T_{\Delta}})\Delta_{j-\frac{1}{2}}^{n}-\rho_{j+1}^{n}\Delta_{j+\frac{3}{2}}^{n-T_{\Delta}}+\rho_{j-1}^{n}\Delta_{j-\frac{1}{2}}^{n-T_{\Delta}})
\end{align}
The Total Variation at time $t_{n+1}$ (denoted by $TV(\rho^{n+1}_{\Delta})$) is given as $\sum_j|\Delta_{j+\frac{1}{2}}^n|$. For the next step and with \eqref{Diff} as well as shifting the indices, we get 
\begin{align}
\sum_j|\Delta_{j+\frac{1}{2}}^{n+1}|&=\sum_j|\rho_{j+1}^{n+1}-\rho_{j}^{n+1}|\nonumber\\
=\sum_j(1&+2|\frac{\Delta t}{2\Delta x}\rho_{j+1}^{n-T_{\Delta}}|+2|\frac{\Delta t}{2\Delta x}\rho_{j+1}^{n}|)\max\{|\Delta_{j+\frac{1}{2}}^{n}|,|\Delta_{j+\frac{1}{2}}^{n-T_{\Delta}}|\}.
\end{align}
If we introduce a CFL-condition of the type $\Delta t\leq\frac{2\Delta x}{\max\{|\rho_j^{n-T_{\Delta}}|,|\rho_j^{n}|\}}$ (i.e.\ not as strong as above), the dependency on $j$ disappears and we get
$$(5+\frac{1}{\max\{|\rho_j^{n-T_{\Delta}}|,|\rho_j^{n}|\}})\sum_j\max\{|\Delta_{j+\frac{1}{2}}^{n}|,|\Delta_{j+\frac{1}{2}}^{n-T_{\Delta}}|\}.$$
By assumption, $TV(\rho_{\Delta}^n)$ and $TV(\rho_{\Delta}^{n-T_{\Delta}})$ are finite, so we can estimate further
\begin{equation}
TV(\rho_{\Delta}^{n+1})\leq2(5+\frac{1}{\max\{|\rho_j^{n-T_{\Delta}}|,|\rho_j^{n}|\}})\max\{TV(\rho_{\Delta}^n),TV(\rho_{\Delta}^{n-T_{\Delta}})\}.
\end{equation} 
For the estimate in time, we look at 
$$\sum_j|\rho_{j}^{n+1}-\rho_{j}^{n}|.$$
By simply plugging in \eqref{AlteredLF} and using the same CFL-condition, we can write
\begin{align*}
\sum_j|\rho_{j}^{n+1}-\rho_{j}^{n}|&\leq\sum_j\frac{1}{2}(|\rho_{j+1}^{n}|+|\rho_{j-1}^{n}|)+\frac{\Delta t}{2\Delta x}(|V(\rho_{j+1}^{n-T_{\Delta}})||\rho_{j+1}^{n}|+|V(\rho_{j-1}^{n-T_{\Delta}})||\rho_{j-1}^{n}|)+|\rho_j^n|\\
&\leq\sum_j2\max\{|\rho_j^{n-T_{\Delta}}|,|\rho_j^{n}|\}+2+2\max\{|\rho_j^{n-T_{\Delta}}|,|\rho_j^{n}|\}\\
&=\sum_j4\max\{|\rho_j^{n-T_{\Delta}}|,|\rho_j^{n}|\}+2.
\end{align*}
We now have a BV-Bound in space as well as in time, which gives us all the desired  BV estimates. 

\subsubsection{Time span}

With the CFL-condition we introduced, the time step can become smaller every step, since we have no maximum principle. Here, we want to see if we can actually reach every time horizon. Therefore, we plug our $L^{\infty}$-bound into the CFL and get
$$\Delta t\leq\frac{\Delta x}{\max\{|\rho^n|,|\rho^{n-T_{\Delta}}|\}}\leq\frac{\Delta x}{(\frac{3}{2})^n||\rho^0||_{L^{\infty}}}.$$
Now, the time horizon we reach with $n$ time steps is given by
$$t_n=\sum_{i=1}^n\Delta t=\frac{\Delta x}{||\rho^0||_{L^{\infty}}}\sum_{i=1}^n(\frac{2}{3})^i.$$
So for infinite time steps $n\to \infty$, 
we end up with a geometric series which converges to $3\frac{\Delta x}{||\rho^0||_{L^{\infty}}}$. That is the time horizon we can guarantee with this estimate. \newline
With the alternative estimate that depends on $T$, the time horizon we can reach is dependent also on $T$. With $n$ time steps, we can now reach the time horizon
$$t_n=\Delta x\sum_{i=1}^n\frac{1}{\rho^0_{\max}+i\frac{1}{2}||\partial_t\rho||T}.$$
This is basically a harmonic series shifted and with a factor, but it is divergent to $\infty$ if the factor is not zero. This also means, that for small $T$, the steps can be larger. Furthermore, with this second estimate, we can guarantee to reach every time horizon.

\section{Numerical results}
This section is devoted to the numerical simulation results for the model presented above, focusing in particular on the S\&G waves phenomenon, a typical feature of congested traffic, detected when vehicles stop and restart without any apparent reason, generating a wave that travels backward with respect to the cars' trajectories.

Starting from empirical observations and the work done in \cite{cristiani2019interface,zhao2017TRB}, let us assume the velocity function as follow: 
\begin{equation}\label{eq:vel_num}
V(\rho)= \begin{cases}
V_{max} & \qquad \rho\le \rho_f \\
\alpha\ ( \frac{1}{\rho} - \frac{1}{\rho_c} )  & \qquad \rho_f<\rho<\rho_c  \\
0 & \qquad \rho\ge \rho_c
\end{cases}
\end{equation}
where $\alpha>0$ is a parameter, $\rho_c\in (0,\rho_{max}]$ and $\rho_f \in [0,\rho_{max})$ are two density thresholds. In particular $\rho_c$ represents the so-called \textit{safe distance} at the macroscopic level: if the density is higher than $\rho_c$, vehicles do not respect the safe distance so the desired velocity has to be 0, indeed they should stop. 

On the other hand, if the density is very low, which means that vehicles are far enough from each others, the desired velocity is the maximum one.

Note that \eqref{eq:vel_num} respects the hypothesis $|V(\rho)|\le |\rho_{max}|$ choosing $V_{max}=1$. Moreover, depending on the choice made for $\alpha$, the velocity function can be discontinuous.

For the discretization, let us assume the space interval $[a,b]=[0,1]$, $\Dx=0.02$ and periodic boundary conditions. Moreover, the time step $\Dt$ is chosen in such a way the CFL condition is satisfied.
The density thresholds are $\rho_c=0.75$ and $\rho_f=0.2$ as real data suggests.
 The delay term depends on the CFL condition and the initial data, for this reason each numerical test has its time delay interval which ensures the reliability of the model, i.e. $\rho\le \rho_{max}$. However, in general, one has to assume $T_\Delta$ one order of magnitude greater than $\Dt$ to see the effect, i.e. $T_\Delta \approx (10\Dt,20\Dt)$.

\subsection{Backward propagation of Stop \& Go waves}
In order to be more comprehensive as possible in reproducing S\&G waves, let us describe first the backward propagation of the perturbation. After that we will focus also on the triggering of this phenomenon.

\subsubsection*{Test 0}

In order to point out the crucial role played by the delay in this framework, let us compare the evolution of the density obtained with the delayed model and the classical LWR model, or, in other words, when $T_\Delta=0$.

Assume as initial data $\rho^0(x)=\frac{5}{8}+\frac{1}{8}\sin{(2\pi x)}$. Moreover the time step is $\Dt=0.01$ and the delay $T_\Delta=15\Dt$ in the delayed case.

\begin{figure}[htb!]
	\centering
	\includegraphics[width=0.49\textwidth]{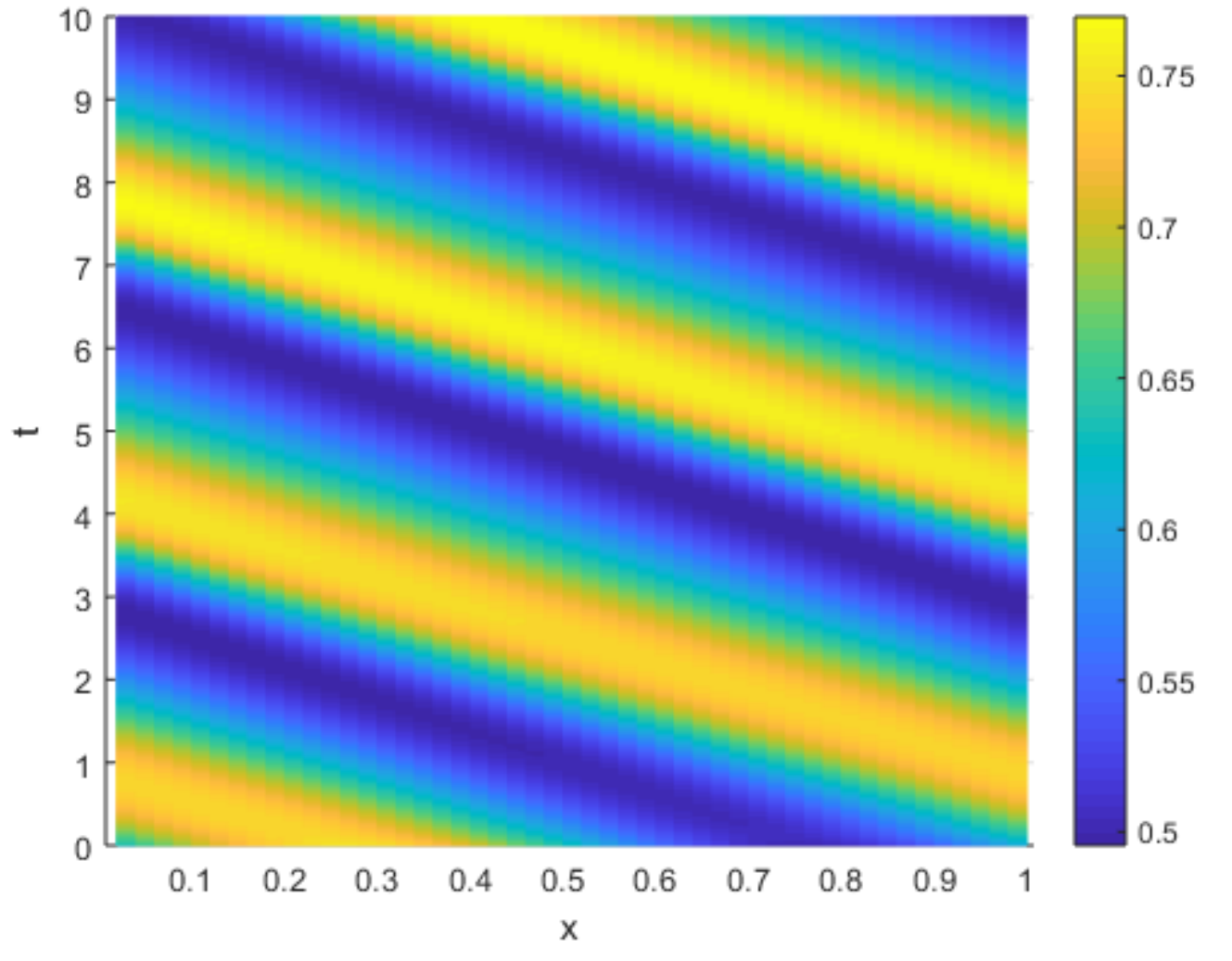}
	\includegraphics[width=0.49\textwidth]{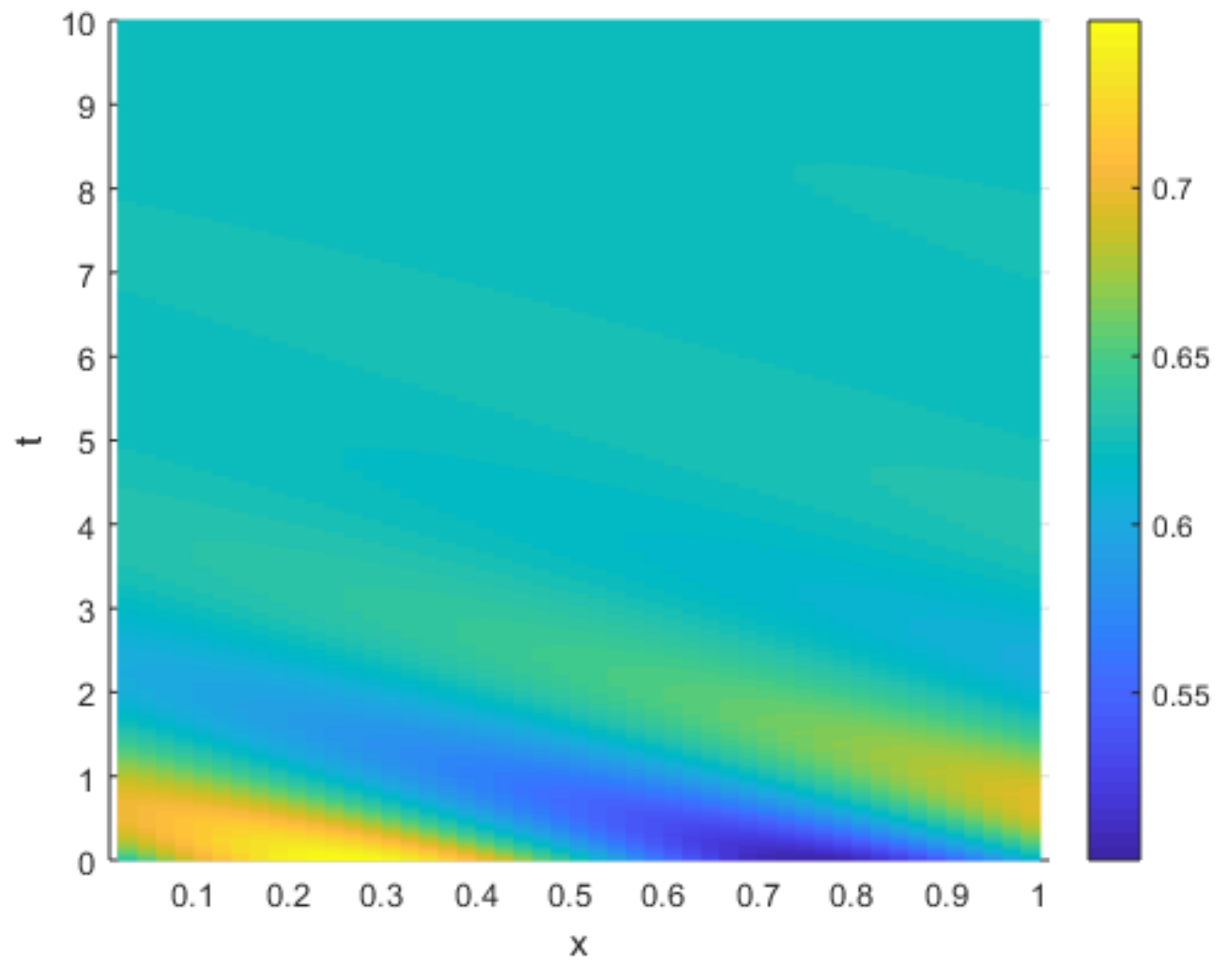}
		\caption{Comparing the density evolution computed by the delayed model (left) and the LWR model (right).}
		\label{fig:t0}
\end{figure}

It is evident how the LWR model smears out the perturbations in the initial data and after a certain time the density becomes constant on the whole road, see Fig.\ \ref{fig:t0}(right). On the other hand, the delayed model preserves the perturbations and also makes them increase as usually happens in traffic evolution, Fig.\ \ref{fig:t0}(left).

\begin{remark}
 We have to be very careful in choosing the delay term. Indeed if the delay is too small, we recover a situation very similar to the LWR model, but, on the other hand, if the delay is too high, i.e.\ $T_{\Delta}=18 \Dt$ , the hypothesis on the model are no longer satisfied and the density grows more than 1, so the model has no sense anymore, see Fig.\ \ref{fig:t0_high_del}. 

\begin{figure}[htb!]
	\centering
	\includegraphics[width=0.49\textwidth]{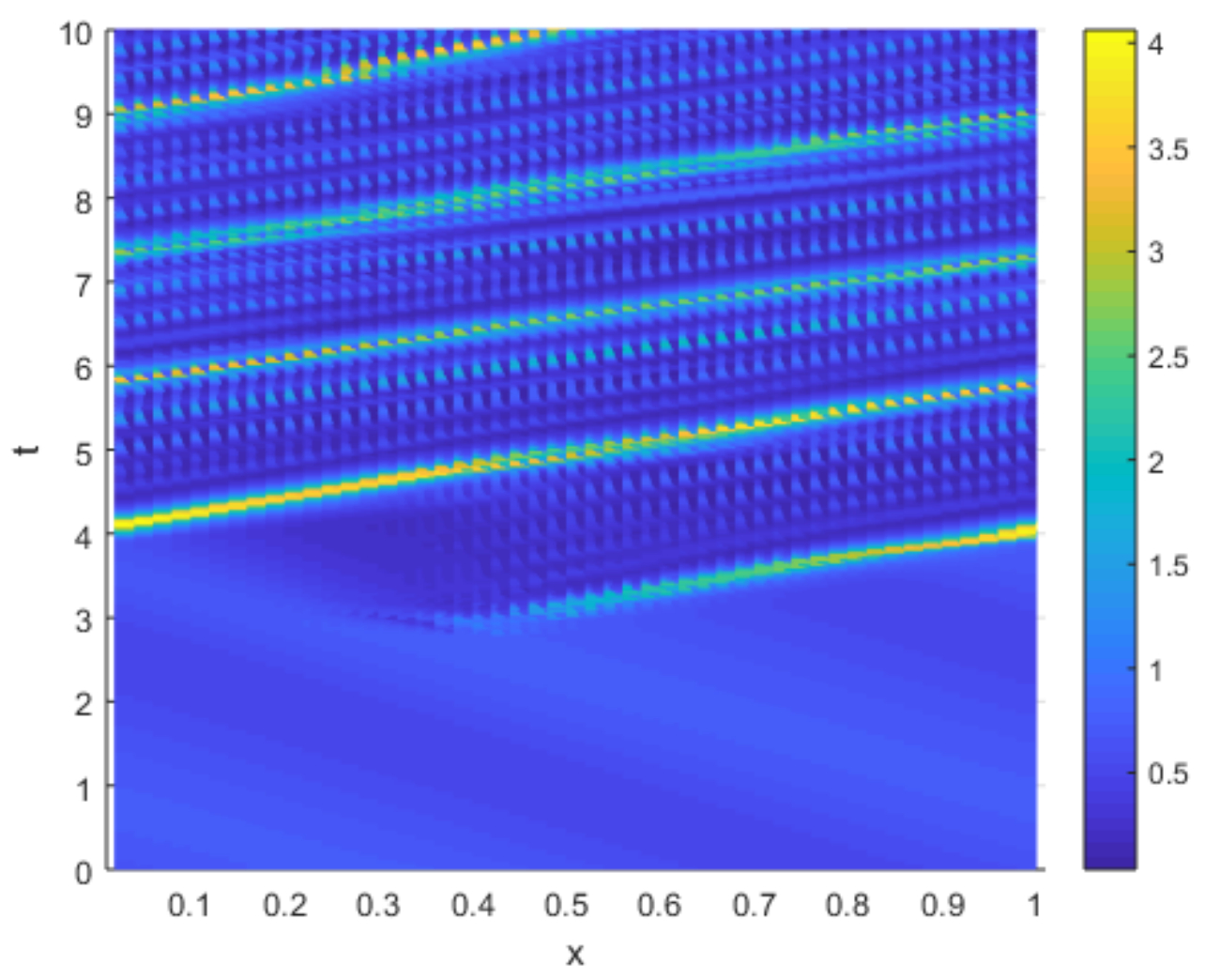}
	\includegraphics[width=0.48\textwidth]{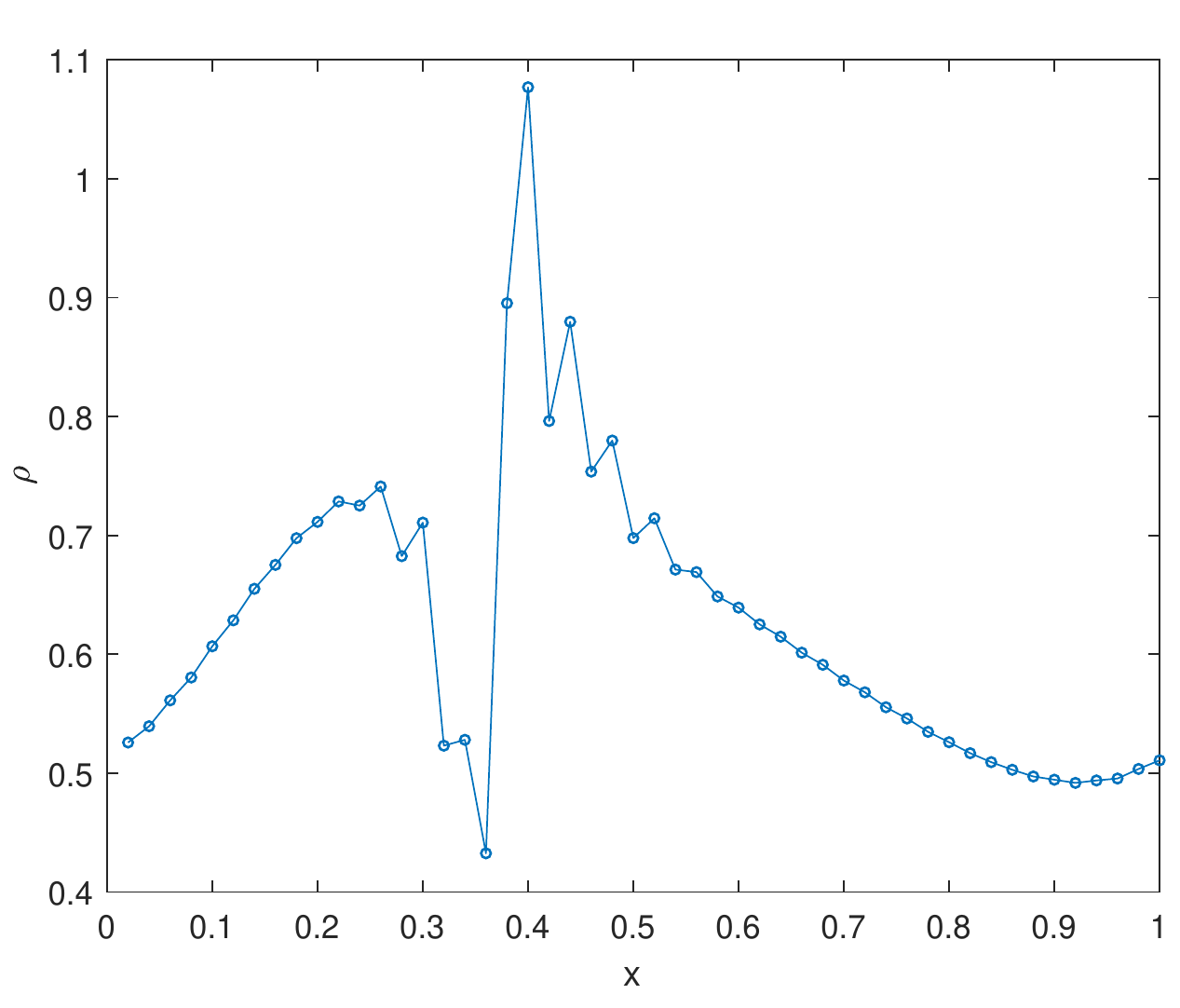}
		\caption{Density evolution and profile, at time $T=\frac{1}{3}T_f$, in case of a too high delay.}
		\label{fig:t0_high_del}
\end{figure}
In the following, we will consider the delay as the maximum allowed by the model feasibility. 
\end{remark}

\subsubsection*{Test 1}

In this numerical simulation we want to reproduce with our model the tests presented in \cite{borsche2018nonlinear}. Starting from the same initial data our aim is to recover a similar behaviour for the density. In \cite{borsche2018nonlinear}, a nonlinear 2-equations discrete velocity model is implemented to compute the density evolution:
\begin{align*}
\partial_t \rho + \partial_x q &=0\\
\partial_t q + \frac{Hq}{1-\rho}\partial_x \rho + \Big(1-\frac{Hq}{1-\rho}\partial_x q  \Big) &=-\frac{1}{\varepsilon}(q-F(\rho)),
\end{align*}
where $q$ is the flux, $F$ is the fundamental diagram and $H$ is a measure for the look ahead and the nonlocality of the equations. Such a model converges to the LWR type equations in the relaxation limit but shows also similarities to the Aw-Rascle model. The main difference to the delayed model is that at the macroscopic level they are composed of a system of two PDEs while the delayed model is described only by a conservation law with a time delay in the velocity term.

Let us assume $\rho^0(x,k)=\frac{5}{8}+\frac{1}{8}\sin{(2k\pi x)}$ for $k=1,2$ and $\Dt=0.01$. We consider $T_{\Delta}=16\Dt$ and $T_{\Delta}=22 \Dt,$ respectively if $k=1,2$.

\begin{figure}[htb!]
	\centering
	\includegraphics[width=0.49\textwidth]{T2_Bor_d1_del32.pdf}
	\includegraphics[width=0.49\textwidth]{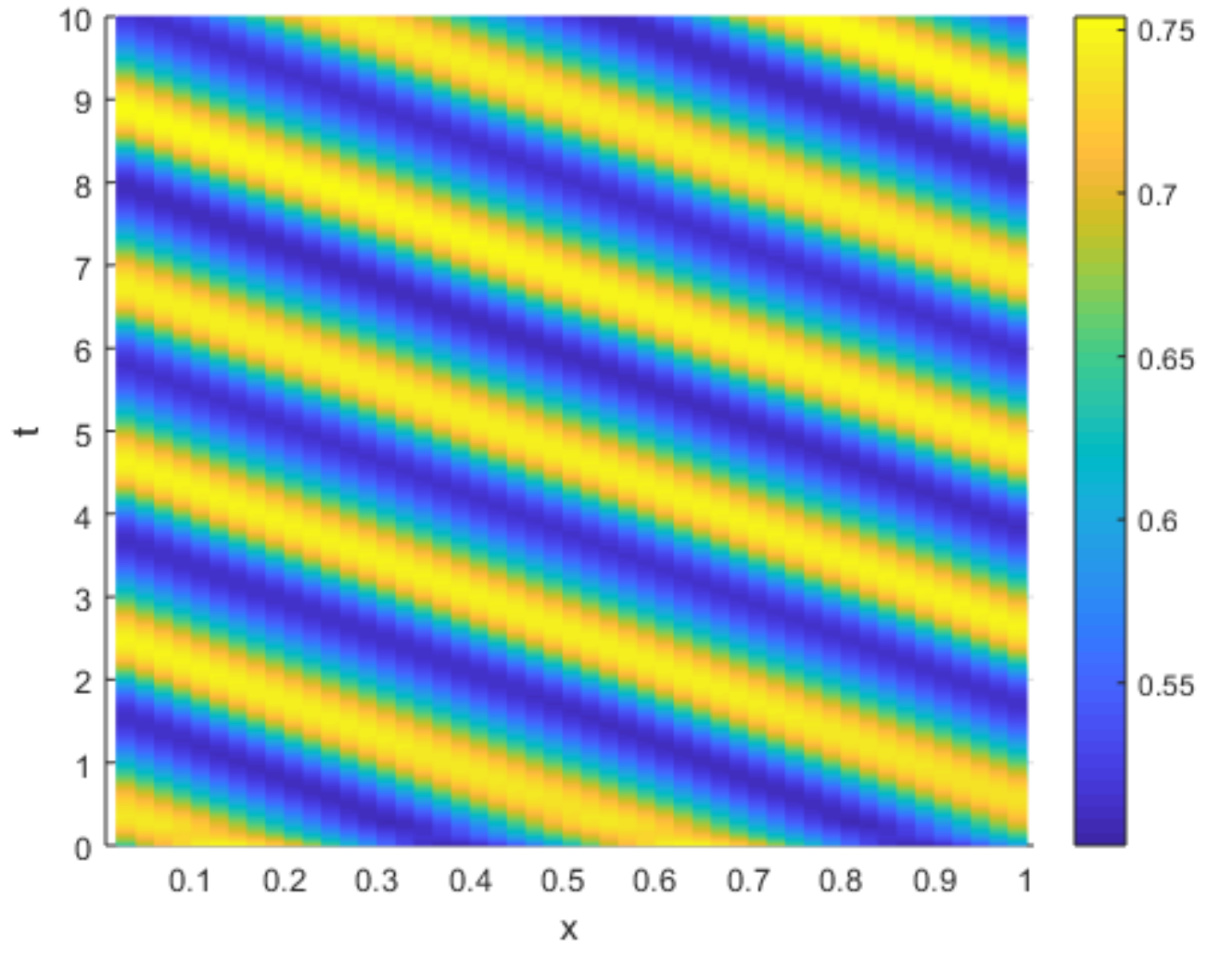}
		\caption{Reproducing the simulation presented in \cite{borsche2018nonlinear}, with $\rho^0(x,1)$ on the left and $\rho^0(x,2)$ on the right.}
		\label{fig:t2_bor}
\end{figure}

In Fig.\ \ref{fig:t2_bor}, the density values on the $(x,t)$-plane are represented by the colors as shown in the colorbar. We obtain persistent waves where the number of waves is directly related to the perturbations in the initial condition, as in \cite{borsche2018nonlinear}.
Numerically, we recover this result if $T_\Delta\in (11\Dt, 16\Dt)$ for $k=1$ and $T_\Delta\in (18\Dt, 22\Dt)$ for $k=2$.

\subsubsection*{Test 2}
 We consider the discrete delayed model \eqref{AlteredLF} with the velocity function \eqref{eq:vel_num} and initial data:
 \begin{equation}
     \rho^0(x)=\begin{cases}
     0.6 & \qquad x< 0.5\\
     0.1 & \qquad x\ge 0.5.
     \end{cases}
 \end{equation}
 Let us assume the time delay as $T_\Delta=10\Dt$ and the time step $\Dt=0.01$ such that the CFL condition is satisfied.
 Moreover, the coefficient $\alpha$ is chosen in such a way as to make the velocity function continuous.
 %
%
Looking at Fig.\ \ref{fig:T1_profile}(left), one can immediately note that the initial slowdown in the first half of the road is increasing in time until vehicles completely stop since $\rho \ge \rho_c$, as the colorbar suggests. This is a typical S\&G wave  behaviour, as we can see also in \cite{tordeux2018traffic}, where they compute the density evolution considering another extension of the LWR model. Indeed they assume that the velocity term depends on $\rho$, the reaction time $\tau$ and the derivative in space of the velocity itself $\partial_x V(\rho)$:

\begin{equation}
    \partial_t\rho + \partial_x \Big( \rho V \Big(\frac{\rho}{1-\tau \partial_x V(\rho)} \Big)\Big)=0.
\end{equation}
 
Investigating the density profile at the end of the simulation, we are able to recognize a well-defined S\&G wave profile, see Fig.\ \ref{fig:T1_profile}(right), as described in \cite{flynn2009PRE}.

\begin{figure}[htb!]
	\centering
	\includegraphics[width=0.49\textwidth]{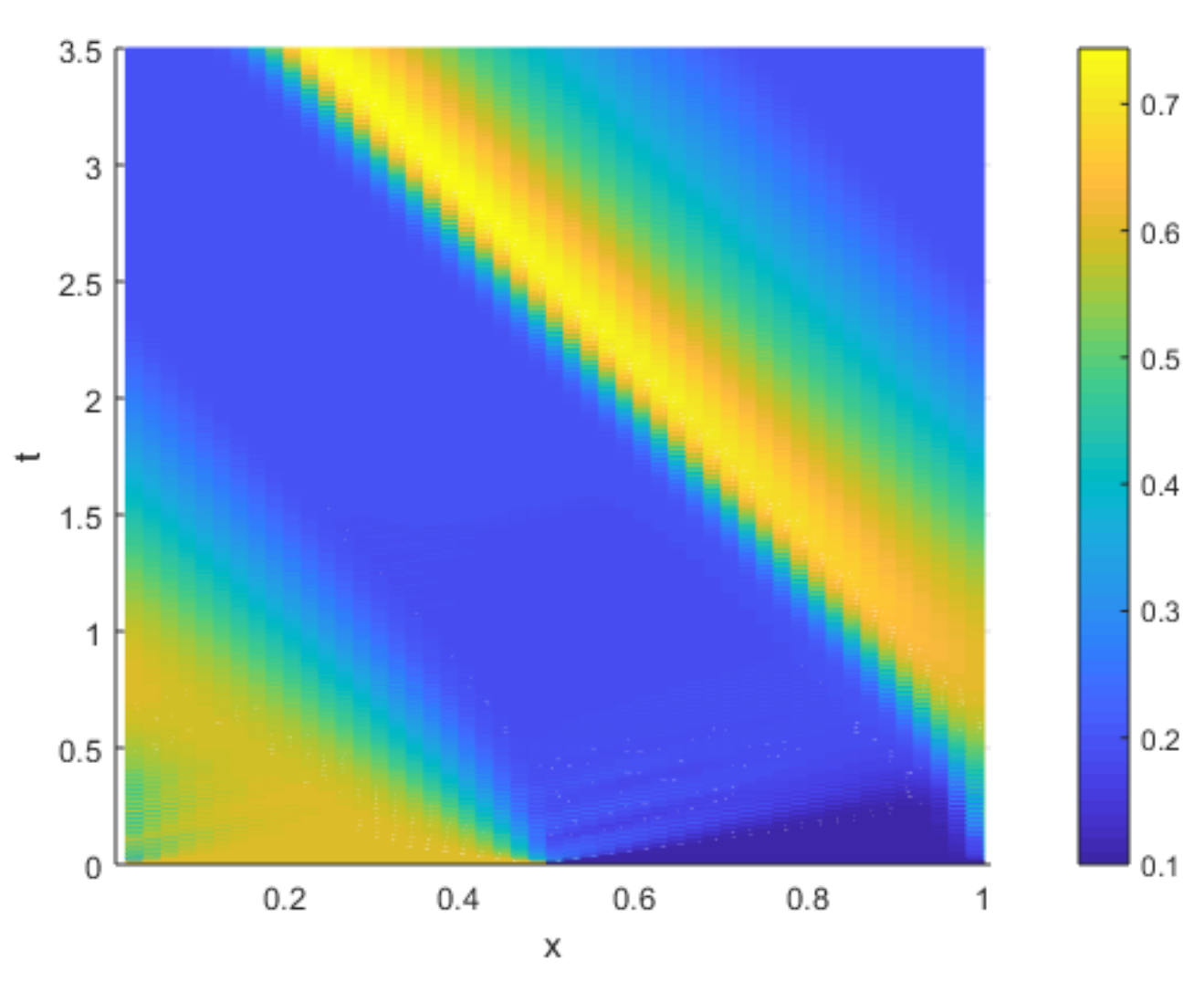}
	\includegraphics[width=0.49\textwidth]{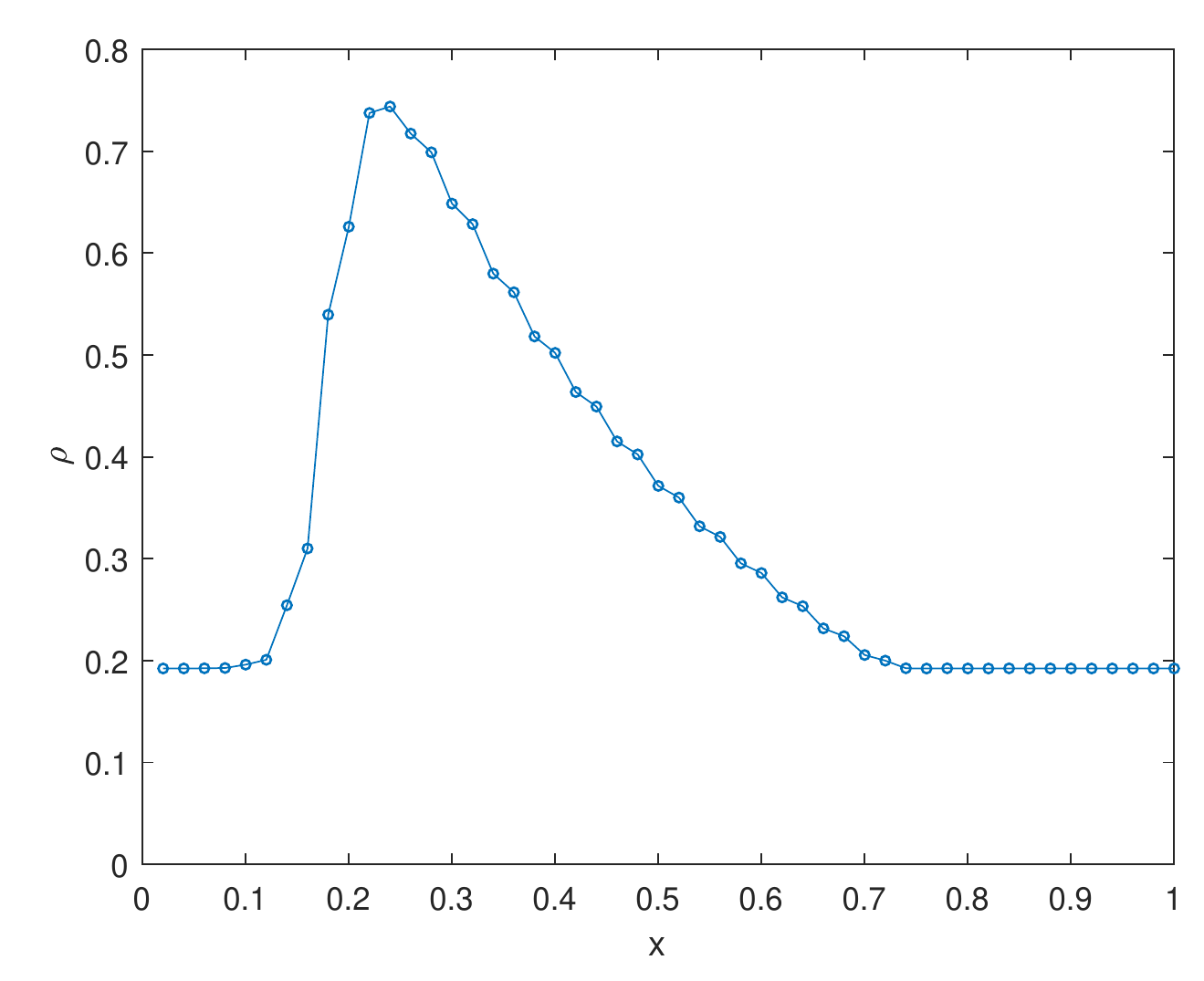}
		\caption{Density values in the $(x,t)$-plane (left) and density profile at time $T=T_f$ (right). }
\label{fig:T1_profile}
\end{figure}

Moreover, let us note that if the delay is smaller, i.e.\ $T_\Delta=4\Dt$, the model is no longer able to preserve the perturbation and the density profile becomes smoother, see Fig.\ \ref{fig:t1_del_low}.

\begin{figure}[htb!]
	\centering
	\includegraphics[width=0.49\textwidth]{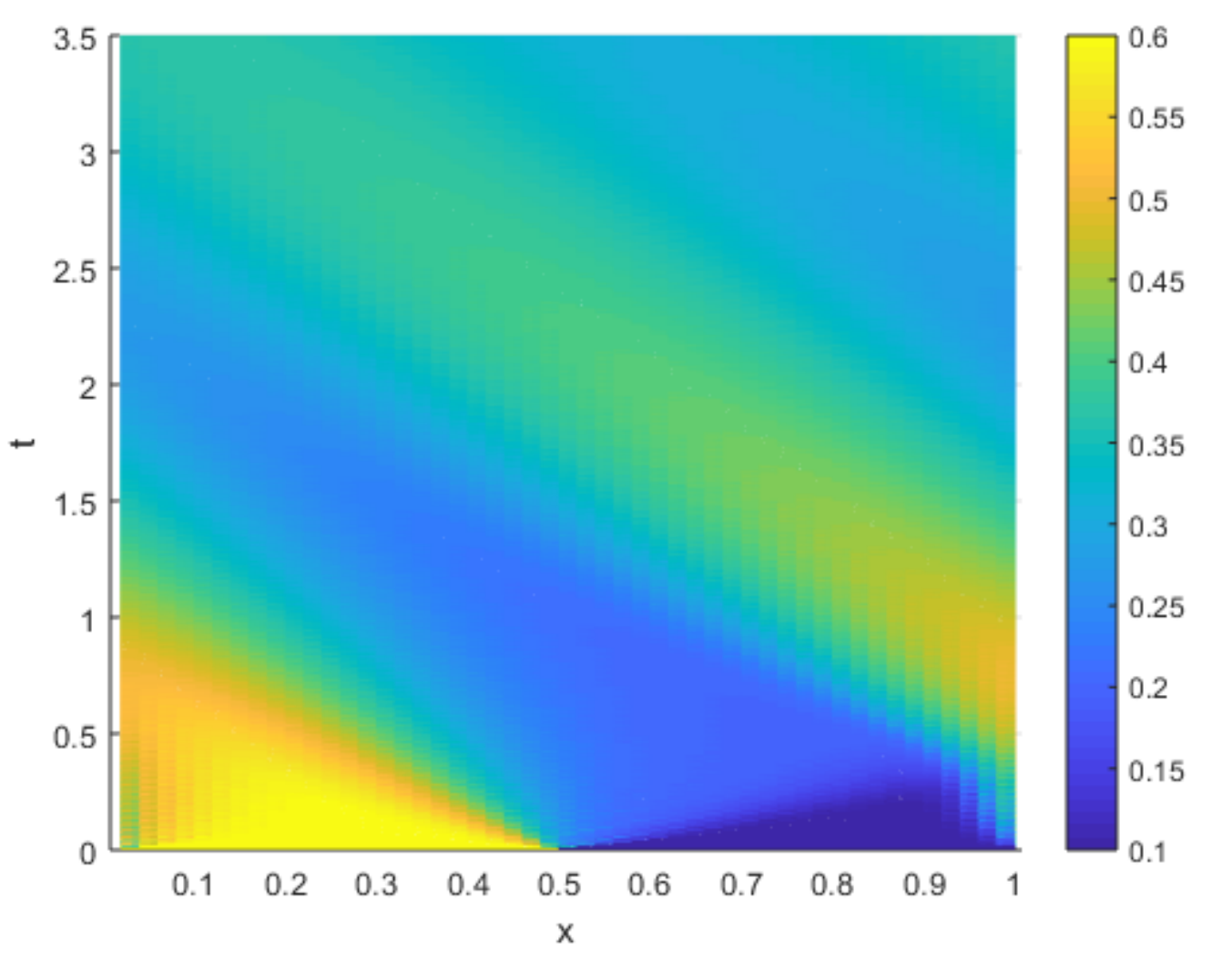}
	\includegraphics[width=0.49\textwidth]{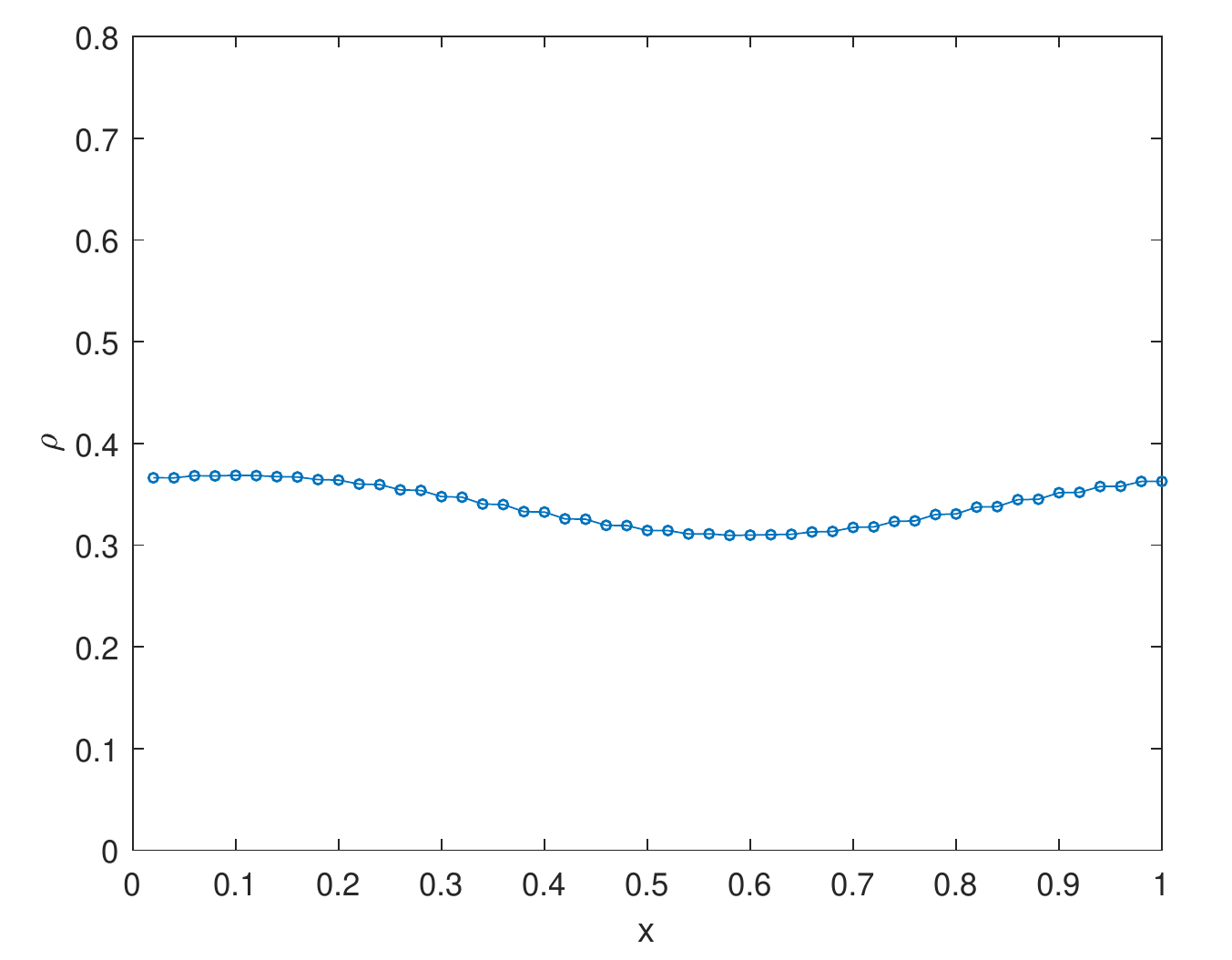}
		\caption{Density values in the $(x,t)$-plane with low delay term, $T_\Delta=4\Dt$, (left) and density profile at $T=T_f$ (right).}
		\label{fig:t1_del_low}
\end{figure}

Choosing the delay term $7 \Dt <T_\Delta< 11 \Dt$, we are able to recover the S\&G behaviour.

\subsection{Triggering of Stop \& Go waves}
Let us now focus on the triggering of S\&G waves. 
Starting from a small perturbation, i.e.\ a slowdown, in the initial data, our aim is to find out if it is possible to recover a S\&G wave. In this direction, let us consider the example presented in \cite{cristiani2019interface}, in which the initial data is given by:
\begin{equation}
    \rho^0(x)=\begin{cases}
    0.35 & \qquad 1.34\le x\le 1.342\\
    0.2 & \qquad elsewhere,
    \end{cases}
\end{equation}
and Dirichlet boundary conditions.
The initial data is modelling a small region, a cell, where vehicles are moving slower than elsewhere and therefore the density is higher in that cell. This slowdown can be caused by the presence of sags, road sections in which gradient changes significantly from downwards to upwards \cite{hoogendoorn2014ITS}, or the presence of a school zone in which the velocity has to be reduced, \cite{song2019second}.

In \cite{cristiani2019interface}, the density evolution is computed by coupling the LWR model with a second order microscopic model, specifically conceived to reproduce S\&G waves.
Instead of switching to multiscale models in which we have to manage the microscopic model too, let us see if we can recover the density evolution with the delayed model.

Assuming $\Dt=0.009$ and $T_\Delta=21\Dt$, we note that the initial perturbation increases and moves backward as the time increases, as it happens in the multiscale case. 
\begin{figure}[htb!]
	\centering
	\includegraphics[width=0.7\textwidth]{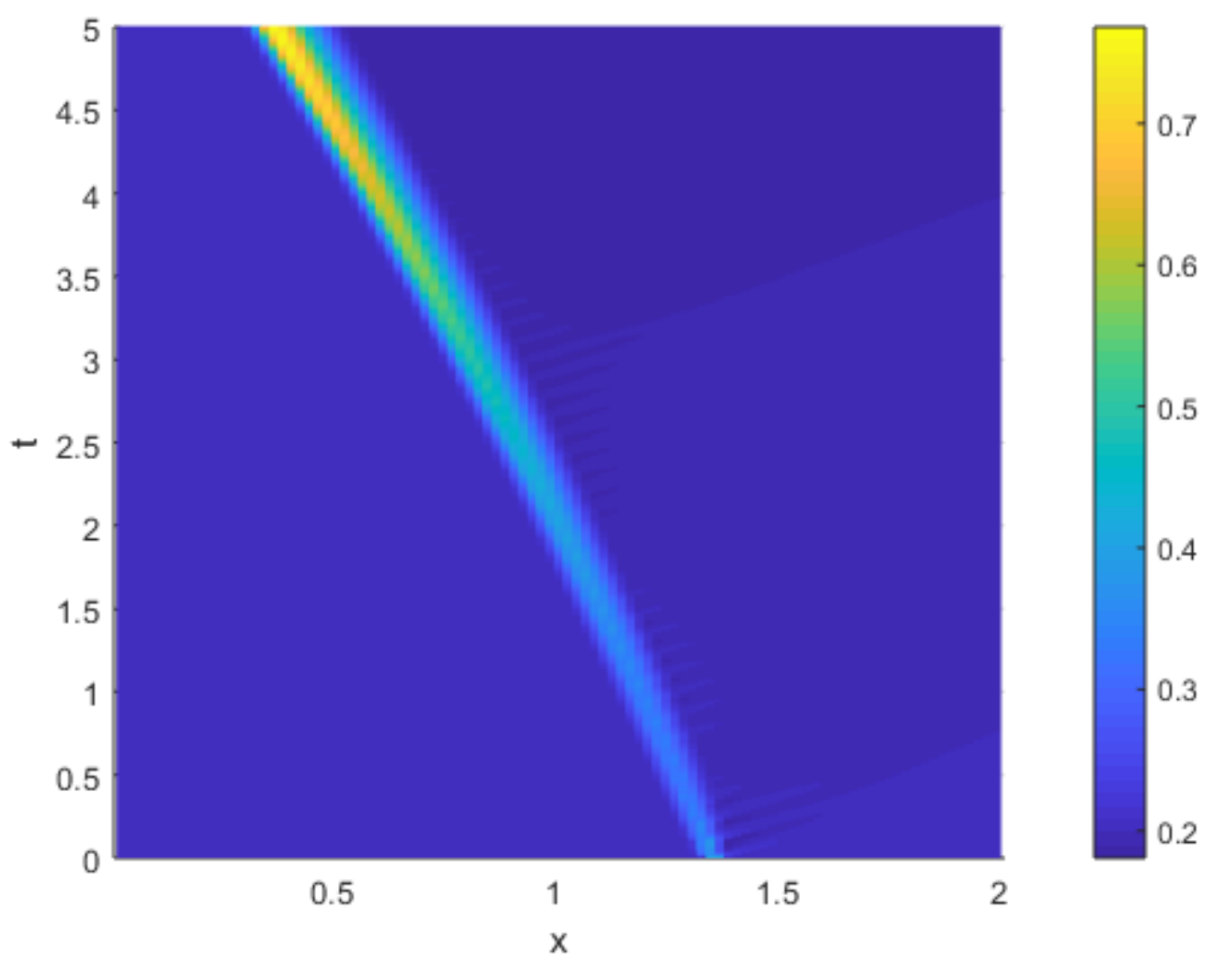}
	\caption{Density values in the $(x,t)$-plane.}
	\label{fig:t3_born}
\end{figure}

Therefore, the delayed model is able to reproduce the triggering of S\&G waves too, not only their backward propagation.

\section*{Conclusion}
In this paper, we have introduced theoretical and numerical properties of the delayed LWR traffic model. While the derivation of the model has been done in~\cite{burger2018derivation}, the numerical behavior of the delayed model has not been studied intensively before. 

Starting from the undelayed scenario, we investigated the theoretical features of the delayed model to point out the numerical properties of the scheme and we proposed an altered Lax-Friedrichs method to compute the evolution of the density.
The key observation therefore is that the delayed model is really able to reproduce Stop \& Go waves for the right choice of parameters. Comparisons to already existing results from the literature also underline this characteristic. 

Future works will include the extension to networks as well as parameter estimation techniques to determine the time delay.

\bibliography{biblio}  
\bibliographystyle{abbrv}

\end{document}